\documentclass[10pt]{amsart}

\usepackage[english]{babel}
\usepackage{amsmath, amssymb}
\usepackage[all]{xy}
\usepackage{todonotes}
\usepackage[normalem]{ulem}

\newtheorem{theorem}{Theorem}[section]
\newtheorem{remark}{Remark}[section]
\newtheorem{proposition}{Proposition}[section]
\newtheorem{cor}{Corollary}[section]
\newtheorem{example}{Example}[section]
\newtheorem{definition}{Definition}[section]

\newcommand{\Z}{{\mathbb{Z}}}
\newcommand{\C}{{\mathbb{C}}}

\newcommand{\Q}{{\mathbb{Q}}}

\newcommand{\pp}{{\mathbb{P}}}

\newcommand{\G}{{\mathcal{G}}}

\newcommand{\MS}{{\mathcal{S}}}

\newcommand{\A}{{\mathcal{A}}}
\newcommand{\Cc}{{\mathcal{C}}}
\newcommand{\cD}{{\mathcal{D}}}
\newcommand{\cL}{{\mathcal{L}}}
\newcommand{\cO}{{\mathcal{O}}}
\newcommand{\cX}{{\mathcal{X}}}

\newcommand{\Sc}{{\mathcal{S}}}

\newcommand{\Fc}{{\mathcal{F}}}
\newcommand{\Tc}{{\mathcal{T}}}
\newcommand{\Dc}{{\mathcal{D}}}

\newcommand{\stab}{\operatorname{stab}}

\begin{document}
\title{Garside elements, inertia and Galois action on braid groups}

\author{F.Callegaro}
\address{F.C.: Dipartimento di matematica, Largo Bruno Pontecorvo, 5 56127 Pisa, Italia\newline 
E-mail address: {\tt  callegaro@dm.unipi.it}}

\author{G.Gaiffi}
\address{G.G.: Dipartimento di matematica, Largo Bruno Pontecorvo, 5 56127 Pisa, Italia\newline 
E-mail address: {\tt  gaiffi@dm.unipi.it}}

\author{P. Lochak}
\address{P.L.:  CNRS and Institut Math\'ematique de Jussieu, 
Universit\'e P. et M. Curie, 4 place Jussieu, 75252 Paris Cedex 05, France\newline 
E-mail address: {\tt  pierre.lochak@imj-prg.fr}}

\begin{abstract}

An important piece of information in the theory of the arithmetic Galois action on 
the geometric fundamental groups of schemes is that divisorial inertia is acted on 
cyclotomically. In this note we describe   a research plan on  this fact in the case of the 
profinite braid groups arising from complex reflection groups, naturally viewing 
them as the geometric fundamental groups of the attending classifying spaces.
We also include the case of the full (non colored) braid groups, whose completed
classifying spaces are Deligne-Mumford stacks rather than schemes.
  
\end{abstract}

\maketitle

\section{Introduction}
\label{sec:introduction}

Some years ago one of us (P.L.) asked whether one could somehow transpose and adapt Grothendieck-Teichm\"uller
theory, which deals with moduli stacks of curves, to the setting of complex braid groups and their classifying
spaces, thus obtaining what could be dubbed a `Grothendieck-Artin theory' and an Artin (or Artin-Brieskorn)
`lego'. We will not discuss here the numerous ins and outs of this suggestion, refering the curious reader to \cite{lo} 
and \cite{mar}. 
This note can be viewed as a warmup (see also 
\cite{mar} and 
\cite{mts}); it deals with the Galois action only and spells out the connection between Garside elements of the complex braid groups and the inertia subgroups attached to the components of the divisors at infinity of the De Concini-Procesi wonderful models associated to the corresponding classifying spaces. We then sketch the first steps of a research plan whose aim is to understand the Galois action on the profinite complex braid groups and as a starting issue to get a list of elements which are acted on cyclotomically by the arithmetic Galois group. 
It seems interesting to inquire whether
or not this list is actually exhaustive (see at the very end of the paper for a precise formulation).

We will need to put together group theoretic, geometric and arithmetic pieces of information; rather than
taking the prerequisites for granted we decided it could be useful to some readers to provide short `reminders',
which of course are designed to be skipped at will. More precisely Section \ref{sec:dpm} recalls the basics of the
construction of the wonderful models whereas Section \ref{sec:center} provides a result on the structure of the center
of the complex braid groups which is needed in the geometric discussion of Sections \ref{sec:inertia_center},
\ref{sec:inertia_divisor} and \ref{sec:quotient}. However the  discussion now has to take care of the additional 
ingredients which are indispensible when dealing with an {\it arithmetic} action, among which we find arithmetic
fundamental groups, profinite completions, fields of definitions, tangential basepoints etc. These are
recalled in the intermediate Section \ref{sec:galois}, in a format adapted to our 
concrete needs. Finally the wrap-up Section \ref{sec:cyclotomic} describes  the expected conclusions of the first part of this research plan (with sketches of the proofs) concerning the Galois action on (divisorial) inertia.

Let us close this introduction with a sheer enumeration of the group theoretic and geometric objects we
will be dealing with. These objects and their notation are quite classical (see especially
\cite{b2, bmr, ddgkm}).
Thus one starts from an irreducible finite complex reflection group $W$;
it gives rise to a central hyperplane arrangement whose complement is denoted $X=X_W$ (sometimes $X(W)$).
The affine variety $X_W$ is $K(\pi,1)$ (see \cite{b2} as an ultimate reference) with topological fundamental group $P(W)=\pi_1^{top}(X_W)$;
see Section \ref{sec:galois} for more on fundamental groups and basepoints. 
The group $P=P(W)$ is the attending pure complex braid group. Moreover $W$ acts 
{\it freely } on $X_W$ with quotient $X_W/W=Y_W=Y$. So we get a Galois \'etale cover 
$X_W\rightarrow Y_W$ with (geometric) Galois group $W$. Because $W$ is a {\it reflection} group, 
$Y_W$ is also {\it affine}, and $\pi_1^{top}(Y_W)=B(W)=B$ is the full complex braid group
($B(W)/P(W) \cong W$). Finally we denote by $\overline{X} =\overline{X}_W$  the (minimal) wonderful model of 
$X=X_W$ whose construction is recalled in Section \ref{sec:dpm}, and by $\cD=\overline{X}\setminus X$
the divisor at infinity, which has strict normal crossings.

{\bf Addendum.} The content of Sections \ref{sec:dpm},~\ref{sec:center}, \ref{sec:inertia_center}, \ref{sec:inertia_divisor}, \ref{sec:quotient} has been extracted as an independent paper (see \cite{cgl16}).

\section{Minimal models} 
\label{sec:dpm}

In 
\cite{dcp1, dcp2} 
De Concini and Procesi introduced and described what they called {\em wonderful models}  associated 
with subspace arrangements. We briefly recall  their construction in the special case of {\it hyperplane} arrangements.
In the next sections we will further specialize to arrangements associated with  complex reflection groups.

\subsection{Irreducible subspaces}
\label{subsec1dpm}
Let $V$ be a complex finite dimensional vector space which we identify with its dual by means of 
a given Hermitian nondegenerate pairing. An \emph{hyperplane arrangement} in $V$ is finite collection $\A$ of affine hyperplanes in $V$. The arrangement $\A$ is a \emph{central arrangement} if $\cap \A \neq \emptyset$. In this case we assume that $O \in \cap \A$.
Let $L(\A)$ be the poset of all possible non-empty intersections of elements of $\A$, ordered by reverse inclusion. We call $\A$ \emph{essential} if the maximal elements of $L(\A)$ are points. In particular if $\A$ is central we have that $\A$ is essential if and only if $\cap \A = \{ O\}.$ 
Let  $\A$ be a central  hyperplane arrangement in  $V$.   
For every subspace $B\subset  V$, we write $B^\perp$ for its orthogonal and denote $\A^{\perp}$ 
the arrangement of {\it lines}  in $V$, dual to $\A$:
$$\A^\perp=\{A^\perp \:|\: A\in \A\};$$
finally let $\Cc_\A$ (or $\Cc(\A)$) be the closure of $\A^{\perp}$ in $V$ under the sum.

\begin{definition}
 Given a subspace $U\in\Cc_\A$, a {\rm decomposition} of $U$ in $\mathbf{\Cc_\A}$ is a collection
$\{U_1,\ldots,U_k\}$ ($k>1$) of non zero subspaces in $\Cc_\A$ such that
\begin{enumerate}
 \item $U=U_1\oplus\cdots\oplus U_k;$
 \item for every subspace $A\in\Cc_\A$ such that $A\subset U$, we have 
$A\cap U_1,\ldots,A\cap U_k \in \Cc_\A$ and $A=\left(A\cap U_1\right)\oplus\cdots\oplus \left(A\cap U_k\right)$.
\end{enumerate}
\end{definition}

\begin{definition}[Irreducible subspace and notation in the case of reflection groups]
 A nonzero subspace $F\in\Cc_\A$ which does not admit a nontrivial decomposition  is called 
{\rm irreducible} and the set of irreducible subspaces is denoted $\Fc_\A$ (or $\Fc (\A)$, or just $\Fc$).  
In the case when $\A=\A_W$ is the hyperplane arrangement associated 
with a complex reflection group $W$ we  write $\Fc_W$ (resp. $\Fc(W)$) instead of $\Fc_{\A_W}$ 
(resp. $\Fc(\A_W)$).
\end{definition}

\begin{remark}
\label{rootirreducibles}
Consider  a root system $\Phi$ in a complexified vector space \(V\) and its associated 
root  arrangement, i.e.  $\A$ is the (complex) hyperplane arrangement defined by the hyperplanes 
orthogonal to the roots in $\Phi$. Then the  building set of irreducibles is the set  of the  subspaces 
spanned by  the irreducible root subsystems of $\Phi$ 
(see \cite{y}).
\end{remark}

\begin{definition}
A subset $\MS\subset \Fc_\A$ is called ($\Fc_\A$-)nested, if
given any subset $\{U_1,\ldots ,U_h \}\subseteq \MS$ (with \(h>1\)) of pairwise non comparable 
elements, we have  $U_1 + \cdots + U_h\notin \mathbf{\Fc_\A}$.
\end{definition}

\begin{example}
Let us consider the case of the symmetric group \(W=S_n\) and let \(\A_{S_n}\) be its corresponding {\em essential} 
arrangement in \(V=\C^n/<(1,1,...,1)>\), i.e. we consider the hyperplanes defined by the equations \(x_i-x_j=0\) in \(V\).  

Then   $\Fc(S_n)$ consists of all the subspaces in \(V\) spanned by the {\rm irreducible} root subsystems, 
that is by the subspaces whose orthogonals are described by  equations of the form 
\(x_{i_1}=x_{i_2}= \cdots =x_{i_k}\) with \(k\geq 2\).  Therefore there  is a one-to-one correspondence between 
the elements of $\Fc(S_n)$ and the subsets of $\{1,\cdots,n\}$ with at least 2 elements:
to an \(A\in \Fc(S_n)\) whose orthogonal is described by the equations  \(x_{i_1}=x_{i_2}= \cdots =x_{i_k}\)
there corresponds the set \(\{i_1,i_2,\ldots, i_k\}\).
As a consequence, a $\Fc(S_n)$-nested set  \(\Sc\) corresponds to a set  of subsets of $\{1,\cdots,n\}$ 
with the property that its elements have cardinality \(\geq 2\) and if \(I\) and \(J\) belong to \(\Sc\) 
then either \(I\cap J=\emptyset\) or one of the two sets is included into the other.       

\end{example}

\begin{example}
Let us consider the real reflection group \(W_{D_n}\)  associated with the root system of type \(D_n\) (\(n\geq 4\)). 
The reflecting hyperplanes have equations \(x_i-x_j=0\) and \(x_i+x_j=0\) in \(V=\C^n\).

The  subspaces    $\Fc(W_{D_n})$  are  all the subspaces in \(V\) spanned by the {\rm irreducible} root subsystems. They  can be partitioned into two families. 
The subspaces in the first family are the {\em strong subspaces} \(\overline{H}_{i_1,i_2,...,i_k}\) whose orthogonals are described by  equations of the form 
\(x_{i_1}=x_{i_2}= \cdots =x_{i_k}=0\) with \(k\geq 3\) (if \(k=2\) this subspace is not irreducible). 
We can represent them by associating to \(\overline{H}_{i_1,i_2,...,i_t}\) the subset \(\{0,i_1,i_2,...,i_k\}\) of \(\{0,1,...,n \}\).

The  second family is made by the {\em weak  subspaces} \(H_{i_1,i_2,...,i_k}(\epsilon_2,...,\epsilon_k)\)
whose orthogonals have   equations of the form 
\(x_{i_1}=\epsilon_2 x_{i_2}=  \cdots =\epsilon_k x_{i_k}\) where \(\epsilon_i=\pm 1\) and  \(k\geq 2\). 

Let us suppose that \(i_1<i_2<\cdots < i_k\); then we can represent these weak subspaces by associating to \(H_{i_1,i_2,...,i_t}(\epsilon_2,...,\epsilon_k)\) the  {\em weighted} subset \(\{i_1, \epsilon_2 i_2,..., \epsilon_k i_k\}\) of \(\{1,...,n\}\). 

According to the representation of the irreducibles by  subsets of \(\{0, 1,....,n\}\)  described here, a $\Fc(W_{D_n})$-nested set is represented by a set \(\{A_1,...,A_m\}\) of (possibly weighted) subsets of \(\{0,...,n\}\) with the following properties:
\begin{itemize}
	\item  the subsets that contain 0  are not weighted; they   are linearly ordered by inclusion; 
	\item the subsets that do not  contain 0 are weighted; 
	\item if in a nested set there is no pair of type \(\{ i,-j\}, \{i, j\}\), then  for any pair of subsets  \(A_i,A_j\), we have that, forgetting their weights, they are one included into the other or  disjoint; if  \(A_i,A_j\) both represent weak subspaces one  included into the other (say \(A_i \subset A_j\)), then their weights must be compatible. This means that, up to the multiplication of all the weights of   \(A_i\) by  \(\pm 1\),  the weights associated to the same numbers must be equal.
	\item in a nested set there may be one (and only one) pair \(\{ i,-j\}, \{i, j\}\) and in this case any other element \(B\) of the nested set  satisfies  (forgetting its weights) \(B\cap\{i,j\}=\emptyset\) or \(\{0,i,j\}\subsetneq B\).

\end{itemize}

\end{example}

\begin{example}
Let us consider the real reflection group \(W_{G_2}\)  associated with the root system  \(G_2\). 
Since this is a two dimensional root system, the irreducible subspaces are the subspaces spanned by the roots and the whole space \(V=\C^2\).

Therefore the $\Fc(G_{D_n})$-nested  are the sets of cardinality 1  made by one irreducible subspace and  the sets of cardinality 2 made by \(V\) and by the subspace spanned by one root.

\end{example}

\subsection{Definition of the models and main properties}
\label{subsec2dpm}

 Let \(\A\) now be a (central) {\it hyperplane}  arrangement in the  complex space  \(V\). We denote its complement by 
$X_\A$ (or $X(\A)$) and again write simply $X_W$ (or $X(W)$) in the case of a complex reflection group $W$.
 Then we can consider the  embedding
	$$ i:X (\A)  \rightarrow  V \times \prod_{D \in \Fc(\A)}^{}\pp(V /D^\perp)$$
where the first coordinate is the inclusion and the map from
$X(\A) $ to $\pp(V /D^\perp )$ is the restriction of the
canonical projection $V\setminus D^\perp \rightarrow \pp(V /D^\perp)$.

\begin{definition}
The minimal  wonderful  model $ \overline{X}(\A)$ is  obtained by  taking the closure of the image of the map $i$.
\end{definition}

\begin{remark} \label{maximalremark} 
Actually in \cite{dcp1} not one but many wonderful models are associated with a given  arrangement 
(see \cite{gs} for a classification in the case of root hyperplane arrangements);  here we will focus  on the minimal one.  
Note that among these models there is always a maximal one, obtained by substituting $\Cc(A)$ for $\Fc(\A)$ 
 in the definition above.
\end{remark}

De Concini and Procesi proved in \cite{dcp1} that the complement $\cD$ of  $X(\A)  $ in $\overline{X}(\A)$ is a divisor 
with strict normal crossings whose irreducible components are naturally in bijective correspondence 
with the elements of \(\Fc(\A)\).They are denoted by \(\cD_F\)  (or $\cD(F)$) for $F\in\Fc(\A)$. 

Next if $\pi$ is the projection of $\overline{X}(\A)$ onto the first component \(  V\),  one observes  
that the restriction of $\pi$ to $X(\A)$ is an isomorphism and $\cD(F)$ can be characterized as the unique 
irreducible component of the divisor at infinity $\cD$ such that $\pi(\cD_F)=F^\perp$.  
A complete characterization of the boundary divisor $\cD$ is then afforded by the observation that, 
if we consider  a collection $\Tc$ of subspaces in $\Fc_\A$,  then
$$\cD_\Tc= \bigcap_{A\in\Tc} \cD_A$$
is non empty if and only if $\Tc$ is \(\Fc_\A\)-nested; moreover in that case $\cD_\Tc$ is smooth and irreducible,
and it is part of the stratification associated with $\cD$.

Nowadays this construction of the De Concini-Procesi wonderful models can be viewed as a  special case of  
other more general constructions which, starting from a `good' stratified variety, produce models by blowing up 
a suitable subset of the strata. Among these constructions we recall the models  described by  Fulton-MacPherson in \cite{fm},
by MacPherson and Procesi in \cite{mp}, by  Li in \cite{li},  by Ulyanov in \cite{u} and by Hu in \cite{hu}. An  interesting survey including 
tropical compactifications can be found in Denham's paper \cite{de}.

From this point of view, one considers $V$ as a variety stratified by the subspaces in $\Fc_\A^\perp$
and the model $\overline{X}(\A)$ is obtained by blowing up the strata in order of {\it increasing} dimensions,
proceeding as follows: First choose an ordering \(A_0, A_1,..., A_k\) of the  subspaces in $\Fc_\A^\perp$
which respects dimensions, i.e. such that $dim(A_i)\le dim(A_j)$ if $i\le j$. Assume that $V\in \Fc_\A^\perp$
(which will always be the case in the sequel) and so that $A_0=\{0\}$.
Then one starts by  blowing up  the stratified variety $V$ at the origin $0$, obtaining
a variety $X_0$;  at the next step  one blows up  $X_0$ along the proper transform of $A_1$
in order to obtain $X_1$, and so on... The end result, after a finite number of steps (=the cardinality of $\Fc_\A$)
is the wonderful model $\overline{X}(\A)$. 

\section{The center of the complex braid groups}
\label{sec:center}

From now on we restrict attention to arrangements arising from complex reflection groups.
As mentioned in the introduction, in order to describe the inertia elements we will need a few pieces of 
information on the centers of the attending braid groups, which we recall in this short group theoretic section. 
Let $W$ be an irreducible finite complex reflection group and let $B=B(W)$ and $P=P(W)$ denote as
usual the associated full and pure braid groups. We write $Z(G)$ for the center of a group $G$.

In \cite{bmr} central elements $\beta \in Z(B)$ and  $\pi \in Z(P)$ were introduced; they are
of infinite order, with $\beta^{|Z(W)|} = \pi$. We recall the following results from \cite{bmr, b2, dmm}:

\begin{theorem}\label{thm:center_b} The center $Z(B(W))$ is infinite cyclic, generated by $\beta$.
\end{theorem}
\begin{theorem}\label{thm:center_pure}\label{thm:center}  
The center $Z(P(W))$ infinite cyclic, generated by $\pi$.
\end{theorem}
\begin{theorem} \label{thm:center_short} There is a short exact sequence:
	\begin{equation} \label{eq:short_center}
	1 \to Z(P(W) ) \to Z(B(W)) \to Z(W ) \to 1.
	\end{equation} 
\end{theorem}

We also recall from \cite{bmr} that the center $Z(W)$ has order $$|Z(W)| = \gcd(d_1 , d_2 , \ldots , d_r )$$ 
where $d_1, \ldots, d_r$ are the degrees of $W$.

Theorems \ref{thm:center_b}, \ref{thm:center_pure} and \ref{thm:center_short} were conjectured  in \cite{bmr} and 
proved there for infinite series $G(de,e,n)$ and rank $2$ cases. Moreover in \cite{bmr} Theorem \ref{thm:center_b} is proven for all Shephard groups. 
Some of the remaining cases of 
Theorem \ref{thm:center_b}  are proved in \cite{b2} and an argument due to Bessis and reported in \cite{dmm} completes the proof of Theorem \ref{thm:center_b}. Theorem \ref{thm:center_short} is proven in \cite{dmm} and together with Theorem \ref{thm:center_b} it implies Theorem \ref{thm:center_pure}.


\begin{remark}
\label{minusidentity}
The center of the irreducible Coxeter groups of type $A_n$, $D_{2n-1}$ and $E_6$ is trivial, 
while it is isomorphic to $\Z/2$ for all the other irreducible Coxeter groups. The cardinality
of the center of the other irreducible finite {\rm complex} reflection groups can be determined using 
the table of the degrees given in \cite{bmr}.
\end{remark}

We recall that if $n$ is the rank of a complex reflection group $W$, the latter is called \emph{well generated} if it can be generated by $n$ reflections.
In general, for a well generated complex braid group $B$, there are many monoids such that $B$ can be presented as a group of fractions of the monoid. In several cases these monoids admit a Garside structure, that in general is not unique.

For example we have the following result for Artin-Tits groups:

\begin{theorem}[{\cite[IX, Prop. 1.29]{ddgkm}}]
Assume that $(W, \Sigma)$ is a Coxeter system of spherical type and $B$ (resp. $B^+$) is the associated Artin-Tits group (resp. monoid). Let $\Delta$ be the
lifting of the longest element of $W$. 
Then $(B^+, \Delta)$ is a Garside monoid, and $B$ is a Garside
group. The element $\Delta$ is the right-lcm of $\Sigma$, which is the atom set of $B^+$, and $\operatorname{Div}(\Delta)$ is
the smallest Garside family of $B^+$ containing $1$.
\end{theorem}
Morover, with respect to this Garside structure, we have:

\begin{proposition}{\cite[IX, Corollary 1.39]{ddgkm}}
The center of an irreducible Artin-Tits group of spherical type
is cyclic, generated by the smallest central power of the Garside element $\Delta$.
\end{proposition}

All the finite Coxeter groups are well generated; actually the only 
irreducible complex reflection groups which are {\it not} well generated are 
\(G(de, e, n)\) for $d \neq 1$ and $e \neq 1$,
some groups of rank 2 (namely $G_7$, $G_{11}$, $G_{12}$, $G_{13}$, $G_{15}$, $G_{19}$ and $G_{22}$), 
and $G_{31}$. Now if $W$ is  well generated, $B(W)$ can be equipped with another Garside structure, unrelated to the previous one: for any Coxeter element $c \in W$ we can define a  
\emph{dual braid monoid} Garside structure on $B(W)$, see \cite[Theorem 8.2]{b2} and $\beta$ appears 
as the {\it smallest central power of the Garside element} with respect to that structure (see \cite[Theorem 12.3]{b2}). 

\begin{example}
In the classical braid group on $n$ strands ${\rm Br}_n$, let $\sigma_i$, $i = 1, \ldots, n-1$ be the standard 
set of generators. The Garside element according 
to the natural Garside structure reads 
$$\Delta = \sigma_1 (\sigma_2 \sigma_1) \cdots (\sigma_{n-1}\cdots \sigma_1),$$
representing a global half twist. Let $c \in \mathfrak{S}_n$ be the Coxeter element $(1,2,\ldots, n)$. The Garside element according to the \emph{dual} braid monoid 
Garside structure is
$$\Delta^*= a_{1,2} \cdots a_{n-1,n}, $$
where $a_{i,j} = \sigma_i \cdots \sigma_{j-2} \sigma_{j-1} {\sigma_{j-2}}^{-1} \cdots \sigma_i^{-1}.$ 
So we can rewrite $\Delta^*$ as
$$\Delta^*= \sigma_{1} \cdots \sigma_{n-1}.$$
Setting the $n$ points at the vertices of a regular $n$-gon, $\Delta^*$ thus represents
a $1/n$-th twist (see for example \cite[I, Sec. 1.3 and IX, Prop. 2.7]{ddgkm}). One thus finds or confirm that $\Delta^2 = (\Delta^*)^n$ 
is a generator of the center of ${\rm Br}_n$. Since $\Delta^2$ belongs to the pure braid group ${\rm PBr}_n$, 
it is also a generator of its center.
\end{example}

\begin{example}
The Artin group of type $D_n$ is the braid group $B(W)$ associated to the reflection group $W= W_{D_n}$, 
also denoted by $G(2,2,n)$ in the Shephard-Todd classification. The group $B(W_{D_n})$ is generated 
by the elements $s_1, s_1', s_2, \ldots, s_{n-1}$ with $s_1, \ldots, s_{n-1}$ and $s_1', s_2, \ldots, s_{n-1}$ satisfying 
the relations of ${\rm Br}_n$, together with $s_1s_1' = s_1's_1$. 

The Garside element is 
$$\Delta = (s_1 s_1' s_2 \cdots s_{n-1})^{n-1}$$ 
and the center of the full braid group
is generated by $\Delta$ if $n$ is even, by $\Delta^2$ if $n$ is odd (see Table 5 in \cite[Appendix I]{bmr}). 
The center of the pure Artin-Tits group $P(W_{D_n})$ is generated by $\Delta^2$.
According to \cite[Section 5.1]{b1} the Garside element corresponding to a dual Garside monoid 
structure for $D_n$ is
$$\Delta^* = (s_1s_1's_3s_5\cdots)(s_2s_4\cdots)$$
and the image of $\Delta^*$ in $W$ is the Coxeter element $c.$
\end{example}

\begin{example}
The Artin group of type $G_2$ is the braid group associated to the reflection group $W_{G_2}$.
The group $B(W_{G_2})$ is generated by two elements $s, t$ satisfying the relation $(st)^3=(ts)^3$. 
The Garside element is $\Delta = (st)^3$ (see Table 5 in \cite[Appendix I]{bmr}) which is also the generator of the 
center. The center of the pure Artin-Tits group $P(W_{G_2})$ is generated by $\Delta^2$, while the Garside element 
corresponding to a dual Garside monoid structure (see \cite[Section 5.1]{b1}) is $\Delta^* = st$, whose image in the Coxeter group is the Coxeter element $c$.
\end{example}

We remark that there is no relation between the natural Garside structure (that is known only for finite type Artin-Tits groups) and the dual Garside structure (that are defined for all well generated complex braid groups). Moreover, several dual Garside structures can be defined, one for every Coxeter element. Nevertheless, the Garside element in those structures are related, since their smallest central power generate the center, that is cyclic. 

\section{Inertia and the center}
\label{sec:inertia_center}

Let again $W$ be an irreducible complex reflection group, $\A=\A_W$ the corresponding hyperplane
arrangement which we assume to be essential (see the definition in Section \ref{subsec1dpm}), in the complex vector space $V$ of dimension $n$. 
We write as usual
$X= X_W= V\setminus \A_W$,
$\overline{X}=\overline{X}_W$ for the associated minimal wonderful model, obtained
by a finite sequence of blowups of $V$ viewed as a stratified variety, along strata with non decreasing dimensions.
We let $P=P(W)=\pi_1^{top}(X_W)$ denote the pure braid group associated with $W$.

In particular, let $X_0$ be the first step in this  process, namely the blowup of the space $V$ 
at its origin 0; let $\Dc_0 \subset X_0$ be the corresponding exceptional divisor. 
We write $\tilde A$ for the proper transform in $X_0$ of a subspace $A \subset V$ and identify $X$ with
the complement in $X_0$ of the proper transforms of the hyperplanes in $\A$ and of $\Dc_0$:
$$X \simeq X_0 \setminus \left( \Dc_0 \cup \bigcup_{H \in \A} \tilde{H}  \right).$$

Using this we denote by $z \in P$ the topological {\it inertia} around the divisor $\Dc_0$. 
It can be viewed as the homotopy class of a counterclockwise loop in
$X_0 \setminus \bigl( \bigcup_{H \in \A} \tilde {H}  \bigr) $ around $\Dc_0$, identified with a loop in $X$ 
(cf. e.g. \cite[Appendix I]{bmr}). Here we use the natural complex orientation of the normal bundle of a hypersurface 
in a complex variety. Note that $z$ is {\it a priori} defined, as it should be, as a {\it conjugacy class} in the pure 
braid group $P$; we did not specify a basepoint for the fundamental group $\pi_1^{top}(X)\simeq P$ 
(see \S \ref{sec:galois} below). However $z$ will be shown to be central in $P$ and so {\it a posteriori} 
it turns out to be well-defined as an element of $P$.
Note as well that $z$ also represents the inertia around  the divisor 
\(\Dc_V\) in \(\overline{X}\):
it is the homotopy 
class of a loop in the big open part of $X_0$, which identifies with that of $\overline{X}$; indeed they both 
 identify with $X$.

 We recall from the definition of blow up (see for example \cite[Chapter 1.4 ]{gr-har}) that $X_0$ is defined as the closure of the image of the map
$$V \setminus \{ O \} \to V \times \pp(V),$$
therefore there is a well-defined projection $\pi: X_0\rightarrow\pp(V)$. It defines a line bundle
$$
\xymatrix{
	\C \ar[r] & X_0\ar[d]^{\pi} \\
	& \pp(V) }
$$
whose 
0-section is precisely the divisor $\Dc_0$. This line bundle is the normal bundle of $\Dc_0$
in $X_0$ and a representative of $z$ is given by a loop around the origin in the fiber at a generic point.

If we fix a hyperplane $H \in \A$, we can  restrict  the line bundle to 
the complement of 
$\pp(H)$ in $\pp(V)$, which is affine. The fibre over any point $[v] \in \pp(V)\setminus \pp(H)$ is a line 
$l \times \{ [v]\}= \{ (\lambda v,[v]) \mid \lambda \in \C \} \subset V \times \pp(V)$. We can fix a translation 
$H+\delta$ of $H$ in $V$ that doesn't contain the origin. Then $H+\delta$ intersects the line $l$ in a point 
$p_{[v]} = l \cap (H+\delta).$
This defines a nonzero section of $\pi$ that trivializes our restricted line bundle and 
we thus get a {\it trivial} bundle:
$$
\xymatrix{
	\C \ar[r] & X_0 \setminus \tilde H\ar[d]^{\pi} \\
	& \pp(V)\setminus \pp(H). }
$$

Now recalling that $X\subset V$, consider the restriction of $\pi$ to the preimage of $\pp(X)$, namely:
$$\pi: X_0 \setminus \bigl( \bigcup_{H \in \A} \tilde {H}  \bigr) \rightarrow\pp(X).$$
From the argument above, since $X$ is contained in the complement of an hyperplane 
(simply choose any of the hyperplanes in the arrangement $\A$), we have that the restriction is actually a trivial line bundle. 

Moreover, since $\Dc_0$ is the $0$-section, we can restrict to $X\rightarrow \pp(X)$ and we obtain the trivial fiber bundle
$$
\xymatrix{
	\C^* \ar[r] & X \ar[d]^{\pi} \\
	& \pp(X ).}
$$
whose fiber is the punctured affine line ($\simeq\C^*$, or ${\mathbb{G}}_m$ in more algebraic notation).

So $X$ factors as $X\simeq \pp(X)\times {\mathbb{G}}_m$ and the inertia element $z$ 
is represented by a nontrivial loop in the second factor. Given the base point $x_0 \in X \subset V$, 
we can choose (cf. \cite[\S 2.A]{bmr}) as a representative of $z$ the map 
$$
t \mapsto ([x_0], e^{2 \pi \imath t}).
$$

This way we have essentially proved

\begin{theorem}
\label{thm:inertia_center}
The inertia element $z=z_W$ generates the center of $P(W)$.
If  the group $B(W)$ is an Artin-Tits group equipped with the classical Garside structure or a well-generated complex braid group equipped  with the dual Garside structure
$z$ represents, up to orientation, 
the smallest power of the Garside element that belongs to  \(P(W)\) and is central.
\end{theorem}
\begin{proof}
This immediately follows  from the factorization $X\simeq \pp(X)\times {\mathbb{G}}_m$, 
the description of $Z(P(W))$ given in Theorem \ref{thm:center_pure}
and the  results on well generated groups recalled in Section \ref{sec:center}.
\end{proof}

The same argument of Theorem \ref{thm:inertia_center} will be used inductively (i.e.~applied to subspaces and subarrangements) in the next section to show the relation between the inertia around the other divisors and the center of the corresponding parabolic subgroups.


\section{Inertia and the divisor at infinity}
\label{sec:inertia_divisor}

With the same setting as above, let $A \in \Fc_W$ be an irreducible subspace and define the parabolic subgroup
$$W_A = \{ w \in W \mid w \mbox{ fixes } A^\perp \mbox{ pointwise}\}.  $$

Noting   that $W_A$ is itself an irreducible complex reflection group with an essential action on $A$ 
we let as usual $\A_{W_A}$ be the associated hyperplane arrangement in $A\subset V$ and by $X_{W_A}=X(W_A)$ its 
complement in $A$. Note that the ambient space $A$ is not made explicit in the notation but  this is 
harmless in the sequel. If we define, as is usual in the theory of hyperplane arrangements:  
$$ (\A_W)_{A^\perp} = \{H \in \A_W \mid A^\perp \subset H\},$$ 
then  the hyperplanes in   $\A_{W_A}$  
can be seen as  the intersections with \(A\) of the hyperplanes of  \((\A_W)_{A^\perp}\). 
Further, we will denote, according to the standard notation used for hyperplane arrangements, by $(\A_W)^{A^\perp}$ 
the hyperplane arrangement in 
$A^\perp\subset V$ defined by the intersections with $A^\perp$ of those hyperplanes of the original
arrangement $\A_W$ which do not contain \(A^\perp\).

With this setup we can now generalize the construction of the previous section and define an inertia generator
(or rather a conjugacy class) attached to any irreducible subspace $A$ of the arrangement $\A_W$:

\begin{definition} 
\label{def:inertia}
The inertia class $z_A\in P(W)$ associated with the divisor $\cD_A\subset \cD$, is the homotopy class of a 
counterclockwise loop around $\cD_A$ in the big open part of $\overline{X}_W$ (which can be identified with $X_W$).
\end{definition}

Now consider the restriction to $X_W=X(W)$ of the natural projection 
\begin{equation}\label{eq:p_ort}
\pi_{A^\perp}: V \to A^\perp.
\end{equation}
Intersecting $X(W)$ with an open tubular neighbourhood of $A^\perp \subset V$ we find that the
fiber over a point of the complement of  \((\A_W)^{A^\perp}\)  in $A^\perp$ is isomorphic to $X(W_A)$: 
if the tubular neighborhood is small enough, the hyperplanes of \(\A\) that do not fix \(A^\perp\) 
are ``far away'' from this fiber.

This yields a natural inclusion 
\begin{equation}\label{eq:incl}
i_A: X(W_A) \hookrightarrow X=X(W).
\end{equation}

Let $z_{W_A}\in Z(P(W_A))\subset P(W_A)$ denote the inertia element as constructed in the previous section,
using $W_A$ (resp. $A$) instead of $W$ (resp. $V$).  More precisely, if we denote by $\cD_{A,W_A}$ the divisor in $\overline{X}(W_A)$  associated  to the maximal element of the building set, that is \(A\),  then $z_{W_A}$ is the inertia around  $\cD_{A,W_A}$.
\begin{remark}
\label{rem:parabolici}
Here and in the sequel we will indicate with \(P(W_A)\) two isomorphic groups: the pure braid group associated with the complex reflection group \(W_A\) and the parabolic subgroup of \(P(W)\) associated with the subspace \(A\). It will be clear from the context which is the group  we are dealing with.
\end{remark}

\begin{theorem}
\label{thm:injection}
The map 
$$(i_A)_*: P(W_A)\rightarrow P(W)$$
induced by the inclusion $i_A: X(W_A) \hookrightarrow X=X(W)$ is {\rm injective}.
It maps the inertia element $z_{W_A}\in P(W_A)$ to $z_A\in P=P(W)$. 
\end{theorem}

\begin{proof}
We start by showing that $(i_A)_*$ is injective. This fact is already known from \cite[Section 2.D]{bmr}, but we reprove it using our notation. To this end consider the sub-arrangement 
$(\A_W)_{A^\perp}$ in \(V\) defined above 
(notice that $(\A_W)_{A^\perp}$ is {\it not} essential). The respective complements
in the ambient space $V$ determine the inclusion  $X(W) \subset X((\A_W)_{A^\perp})$ and 
the composition 
$$X(W_A) \stackrel{i_A}{\hookrightarrow} X(W) \hookrightarrow X((\A_W)_{A^\perp})$$
is easily seen to be a homotopy equivalence, hence induces an isomorphism on fundamental groups.
This shows that \((i_A)_*\) is injective.

Moreover (see also \cite[Section 2.D]{bmr}) the standard generators of $P(W_A)$, which are given by loops in $X(W_A)$ around hyperplanes, 
map via $i_A$ to loops in $X$ around hyperplanes and these determine standard generators of $P(W)$.

Let now $x$ be  a point of the complement of  \((\A_W)^{A^\perp}\)  in $A^\perp$. Let us consider the affine subspace 
$A_x=\pi_{A^\perp}^{-1}(x)\subset V$ and the intersection $U_x= A_x \cap B_x$ with an
open ball $B_x$ centered at $x$, small enough  to avoid the hyperplanes not containing $x$. 
Denote by $\overline{U}_x$ the proper transform of $U_x$ in $\overline{X}$. Then $\overline{U}_x$
is isomorphic to $\overline{X}(W_A)$, since of all the blowups that contribute to  the construction of \(\overline{X}(W)\) 
only the ones that involve subspaces  that contain \(A^\perp\) have an effect on $U_x$. Furthermore,  
the big open parts of the varieties $\overline{U}_x$
and  $\overline{X}(W_A)$ are identified by the projection \(\pi_A\: : \: V\rightarrow A\). In the isomorphism 
mentioned above  the intersection $\overline{U}_x\cap \cD_A$ 
corresponds to  the divisor $\cD_{A,W_A}$ in $\overline{X}(W_A)$.

We then notice that a loop in $\overline{U}_x $ around 
$\overline{U}_x\cap \cD_A$  is also  a  loop in $\overline{X}$ around $\cD_A$, which is tantamount to saying that $i_A$ 
maps a representative of the inertia $z_{W_A} \in P(W_A)$, 
that lies in the big open 
part of \(\overline{X}(W_A)\),  to a loop that is homotopic to a representative of $z_{A} \in P(W)$.
\end{proof}

\begin{remark}
Take another wonderful model associated with the arrangement $\A_W$, for example the maximal one 
(see Remark \ref{maximalremark}), and consider the inertia class $z_A$ around the divisor 
$\cD_A$, which is a component of the divisor at infinity of our given model and is  
associated to a {\it reducible} (i.e. not irreducible) subspace $A$.  Then if $A$ decomposes as a direct
sum of irreducible subspaces, $A = A_1 \oplus A_2 \oplus \cdots \oplus A_k$, its associated inertia class 
can be written as a product of commuting factors:
$$z_A =  z_{A_1} z_{A_2}\cdots z_{A_k}  .$$
\end{remark}

The following statement is an immediate corollary of Theorems \ref{thm:inertia_center} and \ref{thm:injection}:

\begin{cor} 
The inertia class $z_A\in P(W_A)\subset P$ attached to the irreducible divisor $\cD_A\subset \cD\subset \overline{X}$ 
generates the center of the parabolic subgroup $P(W_A)\subset P$. If moreover 
$B(W_A)$ is an Artin-Tits group equipped with the classical Garside structure or a well-generated complex braid group equipped with the dual braid monoid Garside structure,
up to a change of orientation, $z_A$ is (the image via $\pi_1(i_A)$ of) 
the smallest power of the Garside element of $B(W_A)$ that belongs to \(P(W_A)\) and is central.
\end{cor}


\section{Inertia in the quotient model} 
\label{sec:quotient}

Let us now consider the action of \(W\) on the spaces \(X_{W}\) and \(\overline{X}_{W}\). 
We   denote by  $Y_W$ the quotient space  $X_W/W$, which is a smooth variety (or scheme),  and by $\overline{Y}_W=\overline{X}_{W}/W$ the quotient of the  model, 
which in general  is {\it not} smooth; there may be singular points on the divisor at infinity.

So let us focus on this divisor at infinity, i.e. on the quotient of the boundary components of \(\overline{X}_{W}\).
First we notice that \(W\) acts naturally on the building set \(\Fc_W\); 
for every \(A\in \Fc_W\) and $w\in W$ we write $w(\cD_{A})=\cD_{wA}$, 
$W(\cD_{A})= \bigcup_{w \in W}\cD_{w A}$. We denote by \(\cL_A\) the quotient
$$\cL_A= W(\cD_A)/W.$$
Clearly if \(A\) and \(B\) are in the same orbit of  the action of \(W\) on \(\Fc_W\), then 
\(\cL_A=\cL_B\). It turns out that  even in the example of W=\(A_3\) the quotient divisor \(\cL_A\) is not smooth, 
but we will indicate  a dense set of smooth points (see Proposition \ref{prop:springergeneric} below).
The cardinality of the fibers of the projection map 
$$W(\cD_A) \rightarrow \cL_A$$
is determined by the cardinality of  the stabilizers of the points in \(\cD_{ A}\), 
so that we are interested in getting some information about these stabilizers.

Let us start by recalling that  in  \cite{fk03} Feichtner and Kozlov provide a description of these groups.  
In order to describe  the points in  \(\overline{X}_{W}\) they  use the following encoding  
(see also \cite{fm}) which records the information coming from the projections of \(\overline{X}_{W}\) 
onto the factors of the product  \(V\times \prod_{A\in \Fc_W}\pp(V/A^\perp)\).

Every point \(\omega \in \overline{X}_{W}\) is represented by a list:
\[\omega=(x, A_1, l_1,A_2, l_2,...,A_k,l_k)\]
where:
\begin{itemize}
	
	\item \(x\) is the  point in \(V\) given by the image of \(\omega\)  in the projection \(\pi\: : \: \overline{X}_{W}\rightarrow V\);
	\item \(A_1\) is the smallest subspace in \(\Fc_W\) that contains \(x\), and it appears in the list only if  \(A_1\neq V\). If \(A_1=V\) then  the encoding stops here: \(\omega=(x)\);
	\item \(l_1\) is the  line in \(A_1\) given by the image of \(\omega\) in the projection  \(\pi\: : \: \overline{X}_{W}\rightarrow \pp \left(V/(A_1)^\perp \right)\) (identifying \(A_1\) with \(V/(A_1)^\perp\));
	\item \(A_2\) is the smallest subspace in \(\Fc_W\) that contains \(A_1\) and \(l_1\), and it appears in the list only if  \(A_2\neq V\), otherwise the list stops : \(\omega=(x, A_1,l_1)\);
	\item \(l_2\) is the  line in \(A_2\) defined by the image of \(\omega\) in the projection  \(\pi\: : \: \overline{X}_{W}\rightarrow \pp(V/(A_2)^\perp)\) 
	
\end{itemize}

and so on...

\begin{proposition}[see  Proposition 4.2 in \cite{fk03}]
	The stabilizer $stab\ \omega \ $ of the point \(\omega=(x, A_1, l_1,A_2, l_2,...,A_k,l_k)\) is equal to 
	\[ \stab \ x \cap  \stab \ l_1 \cap  \stab \ l_2 \cap \cdots \cap  \stab \ l_k\]
	where by \( \stab \ l_i\) we mean the subgroup of \(W\) that sends \(l_i\) to itself 
(not necessarily fixing \(l_i\) pointwise), i.e. the stabilizer of the point \([l_i]\in \pp(V/(A_i)^\perp)\).
\end{proposition}

It is immediate to check that if \(\omega\) is a {generic point of the divisor \(\Dc_A\), more precisely
if it doesn't lie in an intersection of \(\Dc_A\) with other irreducible components of the boundary, 
it can is representated by a {\it triple} \((x,A, l=l_1)\), with stabilizer  
\[  \stab \ x \cap  \stab \ l=W_A \cap  \stab \ l.\]
In other words the stabilizer of a generic point $\omega = (x,A,l)\in \cD_A$
coincides with the stabilizer in \(W_A\) of \([l]\in \pp(V/(A)^\perp)\).
Since we identify \(\pp(V/(A)^\perp)\) with \(\pp(A)\), 
the problem of describing \( \stab \  \omega\) for a generic $\omega$
is reduced to the study of the stabilizers of the points 
of the projective space \(\pp(A)\) under the action of the parabolic 
subgroup \(W_A\), for every \(A\in \Fc_W\).
The following notion of {\em regular element} of a complex reflection group now comes into play.
\begin{definition}
	\label{Springer-generic}
	Given an irreducible  finite  reflection group \(G\) in a vector space  \(V\)
	we call an element \(g\in G\)  {\em regular}, if it has an eigenvector that does not 
      lie in any of the reflecting hyperplanes of \(G\).  If \(g\) is not a multiple of the identity, 
      we call such an eigenvector a {\em Springer regular} vector of \(V\).
\end{definition}

A classification of regular elements of   irreducible real finite reflection groups has been provided by Springer in \cite{spri}. 
Now  we  say that a generic point  $\omega=(x, A, l)$ in  $\Dc_A$ is a {\em Springer generic point} 
if the line \(l\) is {\it not} spanned by a Springer regular element in \(A\) for the action of \(W_A\). 
Finally a point $y \in \cL_A$ is \emph{Springer generic} if it is the image of a Springer-generic point of $\cD_{A}$.

Now we note that given an irreducible subspace $A$, and \(\omega \in \Dc_A\) Springer generic, 
 its stabilizer  \(W_A\cap  \stab \ l\)  is  trivial if 
the center of $W_A$ is trivial, and otherwise it is 
a cyclic group generated by 
a multiple of the identity in $GL(V)$.
Among  the {\it real} irreducible finite  reflection groups, as recalled in Remark \ref{minusidentity}, 
only $A_n$, $D_{2n-1}$ and $E_6$  have trivial center, while in the other cases  the center is $\Z/2$. 


\begin{proposition}
\label{prop:springergeneric}
For $A\in \Fc_W$, the component $\cL_A\subset \overline{Y}_{W}$ of the divisor 
at infinity is smooth at the Springer generic points.
\end{proposition}
\begin{proof}
Let $\omega=(x, A, l)$ be a Springer-generic point.
The statement is trivial if the stabilizer of $\omega$ is trivial. 
If $\stab \omega$ is not trivial we can assume that $Z(W_A)\simeq \Z/m$; the assumption that $\omega$ 
is Springer generic implies that   
there exists an element $\rho \in W_A$, with $\rho_{|A} = \epsilon_m I_{|A}, \rho_{|A^\perp} = Id_{|A^\perp}$, where $\epsilon_m$ is an $m$-th primitive root of unity and $\stab \omega = \langle \rho \rangle = Z(W_A).$ 
So $\rho$ fixes pointwise the divisor $\cD_A$ in  $\overline{X}_W$. 
Moreover, a point $(x,A,\langle v \rangle), t)$ ($t \in \C$ small enough)
in the normal bundle of $\cD_A$ maps onto a tubular neighborhood of $\cD_A$ via 
$$
((x,A,\langle v \rangle), t) \mapsto x + t  v
$$
with $\rho (x + t  v) = x + \epsilon_m t v$.
Since $\rho$ acts via a multiple of the identity on the normal bundle of $\cD_A$ in a neighborhood of $\omega$, 
the quotient $\cL_A$ is smooth near $\omega$.
\end{proof}

\begin{theorem} \label{thm:center_bw}
Let $W$ be an irreducible 
complex 
reflection group acting on the complex vector space $V$. Let $\zeta_W \in \pi_1(Y_W) = B(W)$ be the inertia 
generator around
a Springer generic point of
$\cL_V$.
The loop $\zeta_W$ is a generator of $Z(B(W)).$ 

If  the group $B(W)$ is an Artin-Tits group equipped with the classical Garside structure or a well-generated 
complex braid group equipped  with the dual Garside structure, then $\zeta_W$ is, up to orientation, the 
smallest central power of the Garside element.
\end{theorem}

\begin{proof} 
Up to homotopy we can assume that $\zeta_W$ is a loop around a point $y \in \cL_V$ which is 
the projection of a Springer generic point $\omega \in \overline{X}_W$. 

Recall that in Theorem \ref{thm:inertia_center} we showed that the inertia $z_W$ around the divisor $\cD_V$ generates 
the center of $P(W)$.  If $\stab \omega $ is trivial, a neighborhood of $y$ in $\overline{Y}_W$ is homeomorphic to a neighborhood of $\omega$ in $\overline{X}_W$ and $\zeta_W$ is the image of $z_W$
via the homomorphism of topological fundamental groups induced by the quotient map $X_W \to Y_W$.
As we are assuming that $\stab \omega$ is trivial, the center
$Z(W)$ is trivial. The short exact sequence \eqref{eq:short_center} of Theorem \ref{thm:center_short} 
yields an isomorphism $Z(P(W)) = Z(B(W))$ and the result follows.
	
Assume now that $\stab \omega$ is not trivial. Since $\omega$ is Springer generic, 
$\stab \ \omega = \langle \rho \rangle$, with $\rho = \epsilon_m Id\in W$ (see above).
We need to show that a) $\zeta_W$ belongs to the center $ Z(B(W))$, and that b) it generates that center.
	
To a) we recall from Section \ref{sec:inertia_center} that there is a trivial bundle 
$\pi:X_W \to \pp(X_W)$ with fiber $\C^*$ determining a decomposition $X_W \simeq \pp(X_W) \times \C^*$.
Hence a loop around $\omega$ in $\overline{Y}_W$ can be represented by a path in $X_W$
$$
t \mapsto ([v], e^{\frac{2\pi \imath t}{m}}) \mbox{ for }t \in [0,1].
$$
In particular we can represent $\zeta_W$ as a path  in $X_W \subset V$
$$
t \mapsto e^{\frac{2\pi \imath t}{m}} v\mbox{ for }t \in [0,1],
$$
where $v$ is not a Springer regular element.
We know from the classical result of \cite{chev} and \cite{st} that $V/W$ is an affine space with coordinates given by homogeneous polynomials on $V$, say $p_1(x), \ldots, p_n(x)$, of degrees $d_1, \ldots, d_n$.
	
Hence $\zeta_W$ is represented by a loop $\gamma$ in $Y_W \subset V/W$ given by
$$
t \mapsto  \gamma(t) = (e^{ \frac{2 d_1 \pi \imath t}{m}}p_1(v), \ldots, e^{ \frac{2 d_n \pi \imath t}{m}} p_n(v)) \mbox{ for }t \in [0,1].
$$ 
Let $\gamma':[0,1] \to Y_W$ be another closed path with the same base point in $Y_W$. We claim  that $\gamma$ and $\gamma'$ commute. In fact the following map $H:[0,1]\times[0,1] \to Y_W$
$$
H(t,t') = (e^{\frac{2 d_1 \pi \imath t}{m}}\gamma_1'(t'), \ldots, e^{ \frac{2 d_n \pi \imath t}{m}}\gamma_n'(t')) \mbox{ for }(t,t') \in [0,1] \times [0,1],
$$
provides a homotopy between $\gamma \circ \gamma'$ and $\gamma' \circ \gamma$. This proves that 
indeed $\zeta_W$ is central.
	
In order to show that it generates the center of $B(W)$, consider again the  short exact sequence 
\eqref{eq:short_center} of Theorem \ref{thm:center_short}. Since $Z(W) = \Z/m$ we get
$$
1 \to Z(P(W) ) \to Z(B(W)) \to \Z/m \to 1.
$$
By construction we know that $z_W$ is a generator of $Z(P(W))$ and it maps to $\zeta_W^m$. Moreover $\zeta_W$ maps to the generator of $\Z/m$. Since $Z(B(W))$ is infinite cyclic, it follows that $\zeta_W$ generates $Z(B(W))$.
\end{proof}


We now observe that the inclusion $i_A: X(W_A) \hookrightarrow X=X(W)$ given in Equation \eqref{eq:incl} 
(see \S \ref{sec:inertia_divisor})
is $W_A$-equivariant, with the $W_A$-action compatible with the inclusion $W_A \subset W$. Hence it induces an inclusion $i_A: Y(W_A) \hookrightarrow Y = Y(W)$ and a corresponding map between the fundamental groups.
As in Section \ref{sec:inertia_divisor} we generalize the definition of inertia for an irreducible subspace $A$. 

\begin{definition} 
\label{def:inertia_braid}
The inertia class $\zeta_A\in B(W)$ associated with the divisor $\cL_A$ is the homotopy class of an oriented 
loop around  a Springer generic point of $\cL_A$ in the big open part of $\overline{Y}_W$ 
(which can be identified with $Y_W$).
\end{definition}

Now let us denote by $\cL_{A,W_A}$ the quotient divisor  of  $\overline{Y}(W_A)$  associated  to the maximal 
element of the building set, namely \(A\);   then we denote by $\zeta_{W_A}$  the inertia around  $\cL_{A,W_A}$.
Below we use \(B(W_A)\) to denote both the braid group associated with the complex reflection group \(W_A\) 
and the (isomorphic) parabolic subgroup of \(B(W)\) associated with the subspace \(A\). It will be clear from 
the context which group we are dealing with (see also Remark \ref{rem:parabolici}).

\begin{theorem}
\label{thm:injection_B}
The map 
$$(i_A)_* = \pi_1(i_A) : B(W_A)\rightarrow B=B(W)$$
induced by the inclusion $i_A: Y(W_A) \hookrightarrow Y=Y(W)$ is \emph{injective}. 
It maps the inertia element $\zeta_{W_A} \in B(W_A)$ to $
\zeta_A\in B$.  Moreover the loop $\zeta_A$ is a generator of $Z(B(W_A)),$ 
the center of the parabolic subgroup $B(W_A)$.

If  the group $B(W_A)$ is an Artin-Tits group equipped with the classical Garside structure 
or a well-generated complex braid group equipped  with the dual Garside structure then 
$\zeta_A$ is, up to orientation, the smallest central power of the Garside element of $B(W_A)$.
\end{theorem}

\begin{proof}
Injectivity follows from the $5$-Lemma (essentially, this rephrases the proof in \cite[Section 2.D]{bmr}): we already know that the natural maps $P(W_A) \to P(W)$ 
and $W_A \to W$ are injective and they fit into the following commutative diagram
$$
\xymatrix{
1\ar[r] & \ar[r] P(W_A)\ar[d]^{(i_A)_*} & B(W_A)\ar[r]\ar[d]^{(i_A)_*} & W_A\ar[r]\ar[d] & 1 \\
1\ar[r] & P(W)\ar[r] & B(W)\ar[r] & W\ar[r] & 1 \\	 
}
$$
where the first and last vertical maps are injective. 

We can choose as a representative of $\zeta_A \in \pi_1(Y_W) = B(W)$ a loop around the projection $y\in \cL_A$ 
of a Springer-generic point $\omega \in \cD_A$. Since the stabilizer $\stab \omega$ 
is the center of $W_A$, the argument used in the proof of Theorem \ref{thm:injection} 
shows that $\zeta_{W_A}\in B(W_A)$ maps to $\zeta_A\in B=B(W)$. The second part of the statement 
then follows from Theorem \ref{thm:center_bw}.

\end{proof}

\section{Galois action and tangential basepoints}
\label{sec:galois}

In this section we start by recalling some facts from the algebraic theory of the fundamental group 
(\S \ref{subsecfundamentalgroup}) an then make precise the use of tangential basepoints in our setting
(\S \ref{subsectangentialbasepoints}). In both case cases we do not seek maximal generality but 
favor statements adapted to our needs. Concerning the theory of the fundamental group, \cite{sga1} remains 
indispensible for a thorough study but several more user friendly introductions are  available.
They are however adapted to the needs of readers with varying backgrounds and it is probably best for
each person to discover the one most in accordance with her or his taste. 
In connection with \S \ref{subsectangentialbasepoints} we refer in particular to \cite{mts}, \cite{na} and \cite{z}.

\subsection{Algebraic fundamental group and arithmetic Galois action}
\label{subsecfundamentalgroup}

This subsection is essentially devoted to writing down and making sense of the fundamental short
exact sequence which gives rise to the (arithmetic) Galois action on the (geometric) fundamental 
group. We start a bit abruptly and add comments after the statement.

Let $X$ be a geometrically connected scheme over the field $k$, $\bar k$
an algebraic closure of $k$, and let $\overline{X}= X \otimes \bar k$. Let $\bar x$ denote a point of
$\overline{X}$, $x$ the image of $\bar x$ via the natural map $\overline{X}\rightarrow X$. Write $\pi_1(X, x)$
(resp. $\pi_1(\overline{X}, \bar x)$) for the fundamental group of $X$ (resp. $\overline{X}$)
based at the point $x$ (resp. $\bar x$). Finally let $Gal(k)$ (or $G_k$) denote the (absolute) 
Galois group of $k$: $G_k=Gal(\bar k/k)$. Then there is a natural short exact sequence:
$$1\rightarrow  \pi_1(\overline{X}, \bar x) \rightarrow \pi_1(X, x) \rightarrow G_k\rightarrow 1, $$
which we call the fundamental short exact sequence (FSES). It gives rise to an outer action
of the arithmetic Galois group $G_k$ on the {\it geometric} fundamental group 
$\pi_1^{geom}(X, \bar x)=\pi_1(\overline{X}, \bar x)$:
$$G_k \rightarrow Out(\pi_1(\overline{X}, \bar x)),$$
where for any group $G$, $Out(G)=Aut(G)/Inn(G)$.
There remains to detail these terse indications.
First we assume that $X$ is a scheme; one can think of a variety defined over the field $\Q$ ($=k$) 
or a finite extension of it. By \cite{mar}, Theorem 1, we know that indeed both $X_W$ and $Y_W=X_W/W$ 
are defined over $\Q$ for any (finite) complex reflection group $W$.
On the other hand we could take above $X$ to be a (say Deligne-Mumford) stack 
rather than a scheme (see \cite{lv} for much more) and indeed we will need a {\it very} particular example 
of this more general case when dealing with the completions ${\overline Y}_W$ of the $Y_W$'s. 
These are naturally complex orbifolds, the analytic avatars of Deligne-Mumford stacks, 
and can be given the structure of  Deligne-Mumford stacks over (small) finite extensions 
of $\Q$. Everything then remains much as above (see \cite{lv}) but we will hardly need any general
theory for the case at hand, especially given the explicit description provided in \S \ref{sec:quotient}.

The algebraic fundamental group $\pi_1(X, x)$ is the automorphism group $Aut(\tilde X/X)$ 
of the `algebraic universal cover' $\tilde X$, defined as the (inverse) limit of the {\it finite}
\'etale covers $Y/X$. More on basepoints below. If $X$ is normal, which will always be the case here, 
{\it a fortiori} smooth, `\'etale' is just the same as `unramified'. An \'etale cover of a field $k$
is a finite separable extension. We assume from now on that $k$ has characteristic 0, so 
erase the word `separable' (and identify separable and algebraic closures). We also fix 
a complex embedding $k\hookrightarrow \C$; it is again best to think of $k$ as being $\Q$ 
or a numberfield (a finite extension of $\Q$). It is also quite useful to think of $Gal(k)$ as
$\pi_1(Spec(k))$, where $Spec(k)$ is the one-point scheme whose covers are in one-to-one
correspondence with the finite extensions of $k$; the omitted basepoint corresponds to fixing an
algebraic closure of $k$. 

The scheme $\overline{X}$ can be regarded as a `geometric' version of $X$, since all the purely 
arithmetic covers, i.e. the $X\otimes \ell$ with $\ell/k$ finite, have been effected.  
We often write $X^{geom}$ and $\pi_1^{geom}(X)$ for $X\otimes \bar k$
and $\pi_1(X\otimes \bar k)$ respectively. The (FSES) can then be thought of as the homotopy 
sequence of a fibration: The base is $k$ (more precisely $Spec(k)$), the total space is $X$, the fiber 
is $\overline{X}$. The sequence is truncated because $\pi_2(Spec(k))=0$, a statement which
can be made sense of and proved.

\begin{remark} \label{terminology} 
$\overline{X}=X\otimes \bar k$ appears as a kind of `arithmetic completion' of $X$, hence this traditional 
piece of notation. However $\overline{X}$ often stands for a {\rm geometric} completion of $X$ and here both are
needed. As mentioned already we will use $X^{geom}$ rather than $\overline{X}$ for the `arithmetic completion'.
\end{remark}

The geometric fundamental group $\pi_1^{geom}(X)=\pi_1(X\otimes \bar k, \bar x)$ can often 
be computed, but by transcendental methods only. No algebraic computation is available to this day. 
The principle is a comparison result which goes as follows. Recall we assumed that $k$ has 
characteristic 0, indeed can be regarded as a subfield of $\C$; $\bar k$ is the algebraic
closure of $k$ in $\C$. One first shows that the extension between the two algebraically
closed fields $\bar k$ and $\C$ does not alter the fundamental group: 
$\pi_1(X\otimes \C)=\pi_1(X\otimes \bar k)$. Next $X\otimes \C=X_\C$ is a $\C$-scheme,
to which one associates a complex analytic space which we denote $X^{an}$
(passing from the Zariski to the finer analytic topology). One can then 
compute $\pi_1^{top}(X^{an})$, the ordinary, topological fundamental group of the analytic
space $X^{an}$. Assuming $X$ is reduced (which again will always be the case here),
this is an analytic variety which one simply views as a topological manifold. We have assumed
that $X$ is geometrically connected, meaning $X^{geom}$ is connected and so $X^{an}$
is connected as well. The topological fundamental group $\pi_1^{top}(X^{an})$ is then 
a discrete finally generated group. The GAGA type comparison result mentioned above asserts
that $\pi_1^{geom}(X) (=\pi_1(X_\C))$ is naturally isomorphic to the profinite completion of
the group $\pi_1^{top}(X^{an})$. The newcomer is encouraged to experiment on the very first 
examples, namely take $X$ to be $\pp^1_\Q$ with 1, 2 or 3 points removed, where $\pp^1_\Q$
is the projective line over $\Q$ (i.e. the `rational line' with a point at infinity adjoined).

Now a word about basepoints. Here we use geometric points of $X$, whereas in the next 
subsection we will introduce {\it tangential basepoints}, which have proved useful since their 
introduction by P. Deligne, and are indeed indispensible for a finer analysis of the cyclotomic property 
of the Galois action. As for now, a geometric point is just a map $x: Spec(\Omega) \rightarrow X$ where 
$\Omega$ is an algebraically closed (and finitely generated) field; because of this last property, 
it factors through $X^{geom}$, and we explicitly did it above, introducing 
$\bar x: Spec(\Omega) \rightarrow \overline{X}$.
We also recall that the point $x$ need not be closed, that is $\Omega$ need not be an 
algebraic closure of $k$. In fact the (Zariski) closure of (the image of the map defining) $x$ 
is a subscheme of $X$ which, assuming $x$ is defined by a monomorphism, has dimension
equal to the transcendence degree of $\Omega$ over $k$. In particular, if $X$ is irreducible
one can take as base point its generic point, and for $\Omega$ an algebraic closure of 
the function field $k(X)$. The above extends with relatively minor 
details to the case of Deligne-Mumford stacks (see \cite{lv}).

Let us briefly return to the fundamental short exact sequence. First recall that the outer action
$G_k \rightarrow Out(\pi_1^{geom}(X))$, which is independent of the basepoint, 
is deduced purely formally from the short exact sequence.
It was realized as a great surprise, in the early eighties, that this action is often {\it faithful} 
(i.e. the map above is injective); this is famously the case for the thrice punctured 
projective line, i.e. when $X=\pp^1_\Q\setminus\{0,1,\infty\}$.
Next, the exact is split by using a {\it rational} basepoint. Indeed such a point is given 
by a map $x: Spec(k)\rightarrow X$ and applying the covariant functor $\pi_1$ yields a morphism:
$$\pi_1(x): \pi_1(Spec(k))=G_k\rightarrow \pi_1(X),$$ 
splitting the exact sequence. The justifiably unconvinced reader by this abstract argument is invited 
to get a more concrete and detailed view of the situation. In any case one can then pick a (closed) geometric 
point over $x$ (i.e. choose an algebraic closure of $k$) and use it as a base point. The upshot is that one
obtains a semidirect product defined by a bona fide action: $G_k\rightarrow  Aut(\pi_1^{geom}(X, \bar x))$,
which this time very much depends on the basepoint.
This action is clearly faithful if the outer version is, and again if $k\hookrightarrow \C$,
$\pi_1^{geom}(X)$ is (naturally isomorphic to) the profinite completion of a (in principle computable)
finitely generated discrete group ($=\pi_1^{top}(X^{an})$). The upshot is that the absolute Galois 
group $G_k$ (e.g. for $k=\Q$) has been realized as a subgroup of the (enormous) automorphism group 
of such a group. In closing we add the by now usual observation that the extension to Deligne-Mumford 
stacks is relatively easy (see again \cite{lv}).

\subsection{Tangential basepoints}
\label{subsectangentialbasepoints}
We assume that we are in a `classical' situation, which can be specified as follows:
$X/k$ is a scheme over a field $k$ of characteristic 0, $\overline{X}/k$ a geometric completion of $X$ 
(see Remark \ref{terminology}), with $\cD=\overline{X}\setminus X$ a divisor with strict normal crossings. 
We assume that $\overline{X}$ (and thus $X$) is (finitely presented), integral (i.e. reduced and irreducible) 
and regular (i.e. smooth); in other words it is a smooth variety but of courses $k$ is not assumed to be
-- and usually is not -- agebraically closed. We fix  a complex embedding: $k\hookrightarrow \C$.
All of this applies to $X=X_W$, $W$ a finite complex reflection group. Note that $X_W$ is defined
over $\Q$ whereas one may need to perform field extensions in order to construct 
${\overline X}_W$ by successive bowups of divisors which are not all defined over $\Q$.
So $k$ will be a finite extension of $\Q$, which depends on $W$.

That the divisor at infinity $\cD=\partial \overline{X}=\overline{X}\setminus X$ has strict normal crossings 
means in particular that for $x\in \cD(k)$ there exists an affine neighborhood $U$ of $x$ in $\overline{X}$, 
a regular system of parameters $(t_1,\ldots,t_N)$ at $x$ ($N=dim(X)$) and a subsystem, say 
$(t_1,\ldots,t_n)$ ($1\le n\le N$) such that the local equation of $\cD$ reads $t_1\ldots t_n=0$.
Note that here the $t_i$'s are {\it $k$-rational} functions on $U$, i.e. $U=Spec(\A)$, 
$\A$ is a $k$-algebra, $k\subset \A$ is algebraically closed in $\A$, and $t_i \in \A$. 
In the next subsection this will be made explicit for the case $X=X_W$ and the divisor at infinity
$\cD= {\overline X}_W\setminus X_W$, whose (regular) components are the $\cD_A$'s, 
$A\in \Fc_W$ an irreducible subspace.

One can base the fundamental group $\pi_1(X, x)$ at any point $x$ of $X$, not necessarily a closed point.
From the arithmetical standpoint it is natural to choose for $x$ a $k$-rational point $x\in X(k)$,
defined by a map $x: Spec(k)\rightarrow X$. But it also turned out to be useful to use the $k$-rational 
points `at infinity' as basepoints, i.e. the points of $\cD(k)$. The only problem is that they do {\it not} 
belong to $X$. Recall now that in the topological setting one can base the fundamental group
not only at a point, but in fact at any {\it simply connected} region of the space one is dealing with.
Tangential basepoints provide an algebraic implementation of this remark, the `region' in question
being small, indeed infinitesimal, and close to (part of) the boundary $\cD=\partial \overline{X}$. 

Let us start with a somewhat restricted approach, which in practice is quite useful.
A $k$-rational point $x\in X(k)$ is given by a map $x: Spec(k)\rightarrow X$. 
Of course here we assume that such exist, i.e. $X(k)\ne\emptyset$, in other words there are rational points.
Then this can be enlarged (in very many ways) to a map $Spec(k[\![ t]\!])\rightarrow X$ with a formal parameter 
$t$, representing a formal arc drawn on $X$ from the point $x$. Let now $k(\!( t)\!)=Frac(k[\![ t]\!])$ denote 
the field of Laurent formal series. A (possibly tangential) rational basepoint can be defined as a map:
$$x: Spec(k(\!( t)\!)) \rightarrow X.$$
It determines a map (still denoted $x$) $Spec(k(\!( t)\!)) \rightarrow \overline{X}$ simply by composition 
with the embedding $X\hookrightarrow \overline{X}$.
Since $\overline{X}$ is complete, this extends (valuative criterion) to a map ($x$ yet again): 
$Spec(k[\![ t]\!])\rightarrow \overline{X}$ and the closed point of $Spec(k[\![ t]\!])$ 
(evaluating a series at its constant term) determines
a $k$-rational point $s\in \overline{X}(k)$. Now there are two possibilities: either $s$ lands in $X$ or it lands 
on the boundary $\cD$. In the former case we are back to an ordinary rational basepoint on $X$; in the 
latter we get a genuine tangential basepoint, i.e. the original map $x: Spec(k(\!( t)\!)) \rightarrow X$,
which we rename $\vec s$, corresponds to the generic point of a formal arc drawn on $\overline{X}$, 
based at (the image of) $s\in \cD(k)$; that arc is transverse to $\cD$ at $s$, since its generic 
point is indeed a point of $X$ and not only of $\overline{X}$.

We  recall the original one-dimensional, affine example: the basefield is $\Q$, the varieties 
are given as $X=\pp^1\setminus \{0, 1,\infty\}$, $\overline{X}= \pp^1$, $\cD=\{0,1,\infty\}$.
Choosing the uniformizing paramater $t$ around the point $0$, one has $\Q(X)=\Q(t)$ 
with the obvious embedding $\Q(t)\hookrightarrow \Q((t))$. This yields a map
$Spec(\Q((t)))\rightarrow X$, with target the generic point $\xi$ of $X$ because the situation
is 1-dimensional. But in fact there is an underlying {\it ring} homomorphism: 
$\Q[t]\hookrightarrow \Q[\![ t]\!]$; $X=Spec (\Q[t, 1/t, 1/(1-t)])$ is affine and that ring 
homomorphism determines a map: 
$$Spec(\Q[\![ t]\!])\rightarrow \overline{X} = \pp^1,$$ 
with target the completed local ring at the point 
$0\in \overline{X} = \pp^1$. This tangential basepoint is traditionally denoted $\vec {01}$; the choice
of the parameter $t$ fixes the three punctures at $0$, $1$ and $\infty$ respectively.

The above provides a description of one-dimensional rational basepoints. It is however useful 
to have a slightly more general viewpoint at one's disposal. First from a topological standpoint, 
the `small simply connected region' which one uses as basepoint for the fundamental group needs
not be one-dimensional. Second and this time from an arithmetical viewpoint, it is unnatural
to use closed $k$-rational points of $\cD$ as basepoints, if only because such points 
may very well not even exist in the first place. 

So let us sketch a slightly more intrinsic and general approach (see \cite{mts} for a similar viewpoint).
First one regards $\overline{X}$ as a stratified scheme with generic stratum $X$, 
codimension 1 strata the components of $\cD$, etc. More precisely, 
enumerate the components of $\cD$ as $(\cD_i)_{i\in I}$, with $I$ a finite set (e.g. $I=\Fc_W$ if $X=X_W$).
We do not assume that the $\cD_i$'s are defined over $k$. This means that 
if $\xi_i$ is the generic point of $\cD_i$, $k(\xi_i)$ its function field, the closure of $k$ in
$k(\xi_i)$ is in fact a (finite) extension $\ell_i$ of $k$, and we say that $\cD_i$ is defined over $\ell_i$.

The strata of codimension $n$ are given by the components of the nonempty $n$-fold intersections 
of the $\cD_i$'s, with $n\ge 0$; $n=0$ corresponds to the generic stratum $X$. Here we somewhat
abuse terminology, confusing strata and their closures: $X$ is a stratum but $\cD_i$ is in fact
the closure of a stratum. This will not cause any ambiguity in the sequel. Now let $Y$ denote 
such a stratum, defined as a component of an $n$-fold intersection:
\begin{equation}\label{eq:4}
Y\subset \bigcap_{i\in J} \cD_i
\end{equation}
where $J\subset I$ has cardinal $n>0$ (so that $Y\ne X$). We would like to use a 
(possibly multidimensional) tangential basepoint based at the {\it generic} point $\eta$ of $Y$, 
as a basepoint for the fundamental group $\pi_1(X,\vec \eta)$ of $X$.

To this end one can start from the natural embedding: 
$$k(\xi)\hookrightarrow Frac({\widehat {\cO}}_{\overline{X},\eta})$$
of the field of fractions of $X$ ($\xi$ being the generic point of $X$)
into the completion of the local ring of $\overline{X}$ at $\eta$
(stemming from the fact that a `meromorphic' function on $\overline{X}$ is 
determined by its trace on an infinitesimal neighborhood of $Y$). 
The tangential basepoint at $\eta$ of maximal dimension is then given 
by the map deduced from the field map above, namely:
\begin{equation} \label{eq:5}
Spec(Frac({\widehat {\cO}}_{\overline{X},\eta})) \rightarrow X
\end{equation}
with target the generic point of $X$. But more concretely ${\widehat {\cO}}_{\overline{X},\eta}$ is a 
complete local ring of dimension $d=codim(Y)=N-n$, with residue field $k(\eta)=k(Y)$. 
Since $Y$ is generically smooth as a subscheme of $\overline{X}$, this ring is regular, and by Cohen's 
structure theorem it is indeed {\it non canonically} isomorphic to a ring of formal power series:
\begin{equation} \label{eq:6}
{\widehat {\cO}}_{\overline{X},\eta}\simeq  k(Y)[\![ t_1, \ldots, t_d]\!]
\end{equation}
where the $t_i$'s 
correspond to an infinitesimal parametrization of the (co)normal bundle of $Y\subset\overline{X}$.

In practice determining the tangential basepoint amounts to specifying the parameters $t_i$'s
(see \S \ref{sec:cyclotomic} below for the case $X=X_W$).
It is however often expedient to use a tangential basepoint of lower dimensions, which
amounts to fixing some relations between the $t_i$'s ($i=1,\ldots, d$), letting the number of 
independent parameters drop to $m$ ($1\le m\le d$). If $m=1$, which is quite common,
we are back to the 1-dimensional case detailed above, except that we used $\eta$ rather 
than a point of $\cD(k)$ in the construction, which is more intrinsic and arithmetically significant.
Here a tangential basepoint:
$Spec(k(Y)((t))) \rightarrow X$,
where on the left-hand side we find the field of Laurent formal series over the field $k(\eta)$,
corresponds to a `field of parametrized transverse (or conormal) formal arcs' over a dense open 
of $Y$. This amounts to setting $t_i=\phi_i(t)$ for some 
functions $\phi_i$ e.g. use the diagonal choice $t_i=t$ for all $i$, or for instance 
$\phi_i(t)=t^i$, so that $(t_1,\ldots, t_d)=(t,t^2, \ldots, t^d)$, which is useful in particular
when dealing with configuration spaces of points on surfaces. 

If $X$ is a moduli (or classifying) space for certain objects, determining such a tangential 
basepoint amounts to constructing a family over $k(\eta)((t))$. In particular one can start from 
a `degenerate' family over a dense (affine) open of $U\subset Y \subset \partial \overline{X}$ and try to 
`smear' it over $k(\eta)[\![ t]\!]$ so that the generic point w.r.t. $t$ provides the sought-after
tangential basepoint in the form of a map: $Spec(k(\eta)[\![ t]\!] \rightarrow \overline{X}$.
So $U$ is locally closed in $\overline{X}$ (taking $Y\subset \overline{X}$ to be the 
{\it closure} of a stratum) parametrizing a family of singular objects, and one has to provide 
a one-parameter infinitesimal smooth deformation of that family across the divisor $\cD$. 

This completes our short exposition of the construction of tangential basepoints for (well-behaved)
schemes. Except we did not explain why they deserve to be called basepoints! To this end one has to
show, following \cite{sga1}, that they provide fiber functors in the Galois category of the \'etale
covers of $X$. For this we refer in particular to \cite{mts} (and \cite{z}) and we will make this point more
concrete in the next subsection.

We add a few words about stacks. We will need only a tiny fraction of what follows
in order to treat the case of the ${\overline Y}_W$. The end of this subsection is thus added for
(relative) completeness and can be skipped without impairing the understanding of the sequel.
We confine ourselves to the case of a (locally noetherian) Deligne-Mumford
(DM) stack $\cX/k$ over a field $k$ ($char(k)=0$), which moreover is uniformizable, that is 
there exists an \' etale cover ({\it in the sense of stacks}) of $\cX$ which is a scheme; taking the Galois closure, 
$\cX\simeq[X/G]$ can be seen as the quotient of a scheme $X$ by a finite (algebraic $k$-)group $G$.
A standard reference for this material and much more is \cite{lm}. Now Section 3 of \cite{z} 
sketches an extension of part of the constructions above to the case of DM stacks, 
for a tangential basepoint based at a {\it closed} point of the divisor at infinity and of maximal 
dimension (i.e. with $N=dim(\cX)$ parameters). It does {\it not} however address the issue of rationality, 
which is crucial when studying the arithmetic Galois action. 

First one has to recall the notion of field of definition for a point of a stack $\cX/k$ (see \cite{lv} 
or \cite[Chap. 11]{lm}). Let $\Omega$ be an algebraically closed $k$-field (i.e. equipped with an 
embedding $k\hookrightarrow \Omega$) and let $x\in \cX(\Omega)$, be a (not necessarily closed)
$\Omega$-point, defined by a map $x: Spec(\Omega)\rightarrow \cX$. The role of the residue field 
for schemes is played by the residue {\it gerbe} $\G_x$, defined by the existence of a factorization 
$Spec(\Omega)\rightarrow\G_x\hookrightarrow \cX$ where, as indicated, the second map is a 
monomorphism (thus $\G_x$ `is' the image of the map $x$).
Because $\Omega$ is a field, the structure of $\G_x$ is simple enough; its moduli space has the
form $Spec(\kappa)$ for a field $\kappa$ with automorphism group the automorphism group 
$Aut_X(x)$ of the point $x$; in other words $\G_x\simeq [Spec(\kappa)/Aut(x)]$. 
The $k$-field $\kappa$ is then a close analog of the residue field and the closure $\ell$
of $k$ in $\kappa$ is the field of definition of the point $x$; because $\cX$ is locally
noetherian and Deligne-Mumford,  $\kappa$ is finitely generated over $k$ and $\ell$ is a finite
extension of $k$. 

Now assume that we are back to our basic situation, only with calligraphic letters, i.e. there is
a DM stack $\cX/k$, a completion $\bar{\cX}$, and a divisor with normal crossings 
$\cD= \bar{\cX}\setminus \cX$. All of this makes sense by taking an atlas. In fact,
any DM stack is locally a quotient (see \cite[Chap. 6]{lm} or \cite{lv}) but, as mentioned above, we will 
have to deal only with {\it global} DM quotients, which amounts to saying that  one can find 
an atlas which is an \'etale cover (in particular proper). So we can make this assumption,
although it does not make much of a difference regarding tangential basepoints, which are in essence 
local objects. The only sticky... point lies rather in that for a point $x$ of $\cX$,
which is a $\kappa$-point in the sense that $\G_x\simeq [Spec(\kappa)/Aut(x)]$,
there may {\it not} exist a map $Spec(\kappa)\rightarrow \cX$ representing $x$. In particular
one cannot for instance define a (1-dim) $k$-rational tangential basepoint as a map 
$Spec(k((t))) \rightarrow \cX$, and ditto for the other types of situations we encountered above.
Below we will take a rather easy way out by considering tangential basepoints on the covering
scheme and mapping it down to $\cX$. That works with any atlas; in  particular, if 
$\cX\simeq [X/G]$ one may consider a tangential basepoint on the scheme $X$, using the inverse 
image of the divisor at infinity, and then project it down. For instance if $K/k$ is an extension,
one may consider a map $x: Spec(K((t))) \rightarrow X$, defining a $K$-rational tangential
basepoint on $X$; composing with the projection $X\rightarrow \cX$, one gets a point
$Spec(K((t)))\rightarrow \cX$ and one can then try and determine the field of definition 
of this point. This is in essence what we will do for the ${\overline Y}_W$'s, which fortunately
appear to feature a simple example in this respect.

\section{ Cyclotomic Galois action}
\label{sec:cyclotomic}

In this section we overview the results, with sketches of the proofs,  of  the first step of our research plan. The aim is to produce a list of elements  which are acted on cyclotomically by the arithmetic Galois group. 
Here the slogan goes: the Galois group acts cyclotomically on divisorial inertia. 
It  can be traced back to the early fifties, if not much earlier, but that basic phenomenon has been 
`rediscovered' again and again, sometimes in much lesser generality, e.g. under the heading 
`branch cycle argument'. We start by stating the basic result we will make use of, 
which is  part of Theorem 7.3.1 of \cite{gm}.

The setting is as above, namely one starts with a scheme $X$ over a field $k$ which 
we assume to be of char. 0, so that tameness is not a problem i.e. is trivially satisfied.
(To be precise, assume that $X$ is reduced, regular, and geometrically connected, so
geometrically irreducible.) Then $\overline{X}$ is a completion of $X$ such that the divisor 
at infinity $\cD={\overline{X}}\setminus X$ has strict normal crossings (its components are smooth). 
Of course we have in mind the case $X=X_W$, $\overline{X}={\overline X}_W$.
Let $Y$ be a component of $\cD$, so $Y=\cD_A$ for some $A\in \Fc_W$ if $X=X_W$.
Assume that $Y$ is defined over $k$ i.e. work over a suitable extension of the original
field of definition of $X$ and rename it $k$.
Next consider the formal completion $\widehat X_{/ Y}$, which describes a formal neighboorhood
of $Y$ inside $\overline{X}$ or, in analytic terms, the germs of all orders of the functions on arbitrarily
small neighborhoods of $Y$. Write $\cD'=\cD\setminus Y$. 

One is interested in relating
the groups $\pi_1^\cD(\widehat X_{/ Y})$ and $\pi_1^{\cD'\cap Y}(Y)$. 
In words, the first one is the fundamental group of the formal neighborhood of $Y$ related
to covers which are unramified outside $\cD$ (in particular they can be ramified along $Y$);
the second corresponds to covers of $Y$ which are unramified outside the intersections
$\cD'\cap Y$. Note that in our favourite case, these intersections have an explicit characterization:
for $A, B \in \Fc_W$, $\cD_A$ and $\cD_B$ intersect if and only if $(A,B)$ is nested.
We can now state:

\begin{theorem} 
\label{thm:GM}
(\cite{gm}) There is a short exact sequence

$$1\rightarrow K \rightarrow  \pi_1^\cD(\widehat X_{/Y}) \rightarrow \pi_1^{\cD'\cap Y}(Y) \rightarrow 1,$$

in which the kernel $K$ is a quotient of $\widehat\Z(1)$.

\end{theorem}

Comments: We will actually need only the last fact, which says that the inertia around $Y$ is a
quotient of $\widehat \Z$ and that the action of the Galois group is like on the roots of unity 
(which is the meaning of the Tate twist $(1)$). 
Let us make this more precise. First as a basepoint for the right-hand
group one uses a geometric generic point $\eta$ of $Y$. Specifically, it is given by a map
$$Spec(\Omega)\rightarrow Y,$$
where $\Omega={\overline {k(Y)}}$ is an algebraic closure of the function field of $Y$. The exact 
sequence formally determines an outer action of the right hand group on $K$, which is in fact, 
since $K$ is commutative, a {\it bona fide} action
$$\pi_1^{\cD'\cap Y}(Y,\eta)\rightarrow Aut(K) (\simeq K^\times).$$
Moreover there is a natural surjection ($Gal(\Omega)=Gal(\overline\Omega/\Omega)$)
$$Gal(\Omega)\rightarrow  \pi_1^{\cD'\cap Y}(Y,\eta).$$
By composing the two maps we thus get a morphism
$Gal(\Omega)\rightarrow K^\times.$ The twist $(1)$ denotes the fact that this  
coincides with (or is a quotient of) the action of that Galois group on the roots of unity.

The fact that we only get a {\it quotient} of $\widehat \Z$ for $K$
stems from the fact that there may not exist enough covers of $\widehat X_{/Y}$ so as to realize
an arbitrary ramification index along $Y$. It would be interesting to determine whether or not, in the
case $X=X_W$, one has in fact $K=\widehat\Z(1)$; here $K=K(W,A)$ {\it a priori} depends on $W$ and $A$
such that $Y=\cD_A$. This may be a hard problem but the analog in the case of the moduli
stacks of curves is known to hold true.  Second, the surjectivity in the above exact sequence is quite
subtle, as it depends on the purity theorem, but we will not need it here. It could however be very interesting
in the case at hand as it suggests a recursive structure very much akin to what happens with 
moduli stacks of curves. Indeed $\pi_1^{\cD'\cap \cD_A}(\cD_A)$ is immediately related to 
the arrangement determined by the parabolic subgroup $W_A\subset W$. Finally 
we did not detail the basepoints used for the middle group (see however \cite[\S 5.1.4]{gm})
because we will not use the result exactly under this form.

We could directly apply this result to our case, and we indeed encourage the reader to do so,
but in order to connect this with the previous sections, we first note that there is a natural surjective map
$\widehat X_{/ Y} \rightarrow X$.
(In the case where $X/k$ is the affine line and $Y=\{0\}$, this is the counterpart
-- taking $Spec$'s -- of the injection $k[t] \hookrightarrow k[\![t]\!]$.) Applying the fundamental group
functor we get a Galois equivariant morphism
$$\pi_1^\cD(\widehat X_{/ Y}) \rightarrow \pi_1(X),$$
under which the subgroup $K$ maps to $I\subset \pi_1(X)$, the inertia group attached to the 
divisor $Y$. It is a quotient of $K$, so again a quotient of $\widehat \Z(1)$, but of course $K$ and $I$ 
may very well differ, for global topological reasons. In particular $X$ could be simply connected,
and then $I$ is trivial, whereas $K$ usually is not. In some sense the local invariant $K$ is more 
precise than the global one $I$. 

We now proceed to make basepoints precise, using in particular tangential ones. These will actually be 
the images under the map $\widehat X_{/ Y} \rightarrow X$ of the basepoints used in \cite{gm} (which are not 
tangential, properly speaking, as far as the completion $\widehat X_{/Y} $ is concerned). 
On $Y$ we use as basepoint a geometric generic point $\eta$ as above; as for $X$ we use a tangential 
basepoint attached to $\eta$, as described in \S \ref{subsectangentialbasepoints} (see equations \eqref{eq:5} and \eqref{eq:6} with $d=1$), 
paying attention that this involves the choice of a formal parameter $t$. We denote that basepoint $\vec\eta$.

The arithmetic group $G_k=Gal(k)$ is a quotient of $Gal(\Omega)$ and in order to let the former act on the
geometric inertia group, we need to tensor with $\bar k$ and use the fundamental
exact sequence. We leave the details to the reader (see \cite[\S 3.3]{lv} for a  detailed discussion under similar
but more delicate circumstances) and call the resulting group $I_Y$. The result reads:

\begin{cor} 
\label{cor:cyclotomy}

Let $X/k$, $\overline{X}$ and $Y$ be as above, and assume that the divisor $Y$ is defined over $k$.
Then there is attached a procyclic inertia group
$I_Y\subset \pi^{geom}(X, \vec\eta)$
and the Galois action
$$G_k\rightarrow Aut(I_Y)$$
is cyclotomic.
\hfill $\square$ 

\end{cor}

To put it more concretely, let $z_Y$ be a generator of $I_Y$ and assume that $k$ is a numberfield;
denote as usual by $\chi: G_k\rightarrow \hat \Z^\times$ the cyclotomic character. Then for
any $\sigma\in G_k$, one has $\sigma(z_Y)=z_Y^{\chi(\sigma)}$. It should be emphasized that this
equality is {\it exact}, not only up to conjugacy, and this was the whole point in being careful with the 
choice of basepoints. If one uses arbitrary basepoints, one will get that
$\sigma(z_Y)\sim z_Y^{\chi(\sigma)}$ (conjugacy in $\pi_1^{geom}(X)$)
which only determines a conjugacy class in $\pi_1^{geom}(X)$. Given two (possibly tangential)
basepoints $b_0$ and $b_1$, both rational (i.e. defined over $k$) the respective actions are 
conjugate via the action of the Galois group on a `path' connecting $b_0$ and $b_1$. Here a path
can be taken (almost) topologically if both points are closed rational points (they belong to $X(k)$)
but if not, one has to resort to a more algebraic notion, essentially an isomorphism of fiber functors
in the Galois category $(Et)/X$ of the \' etale covers of $X$. But again, using a generic geometric
tangential basepoint attached to the divisor, one gets an exact cyclotomic action. 

The application to our case is quite straightforward: take $X=X_W$, ${\overline{X}}={\overline X}_W$,
$Y={\cD_A}$ for some $A\in \Fc_W$. We briefly describe a tangential basepoint $\vec \eta_A$ 
as in \S \ref{subsectangentialbasepoints} above. Recall from \cite[\S 3.1]{dcp1}  that ${\overline X}_W$
is covered by a finite number of open patches ${\mathcal U}_{\Sc}^b$, where $\Sc$ varies over the
the maximal nested sets and $b$ over the bases adapted to a given $\Sc$.
Moreover ${\mathcal U}_{\Sc}^b$ is affine (say ${\mathcal U}_{\Sc}^b= Spec({\mathcal A}_\Sc^b)$)
and $\cD_A\cap {\mathcal U}_{\Sc}^b$ is given by $u_A=0$, i.e. it is the zero locus of a certain 
function $u_A$. Now given $A\in \Fc_W$, pick a maximal nested set $\Sc$ containing  $A$
and define $\vec\eta_A$ as in \S \ref{subsectangentialbasepoints}, using 
the algebra $Frac({\mathcal A}_\Sc^b)[\![u_A]\!]$. 

Corollary \ref{cor:cyclotomy} above translates as:

\begin{cor} 
\label{cor:Acyclotomy}

Given $W$ a finite complex reflection group, $A\in \Fc_W$ an irreducible element,
let $\cD_A\subset \cD={\overline X}\setminus X_W$ be the component of the divisor at infinity
labeled by $A$, and let $\pi_1^{geom}(X_W, \vec\eta_A)$ denote the geometric fundamental group
of $X_W$ based at the tangential basepoint $\vec\eta_A$. Let finally
$I_A\subset \pi_1^{geom}(X_W, \vec\eta_A)$ be the inertia group attached to the divisor $\cD_A$.

Then $I_A=\langle z_A \rangle$ is procyclic, generated by $z_A$; it is a quotient of a copy of 
$\widehat \Z (1)$ in the sense that if $k$ denotes the field of definition of $\cD_A$, the action of $Gal(k)$ 
on $I_A$ is cyclotomic. In other word, for any $\sigma \in Gal(k)$, $\sigma(z_A)=z_A^{\chi(\sigma)}$
($\chi$ = the cyclotomic character).

\hfill $\square$ 

\end{cor}

We recall that here $X_W$ is defined over $\Q$, ${\overline X}_W$ is defined over a certain numberfield 
$k_W$ (a finite extension of $\Q$) and the field $k=k_{W,A}$ appearing in  the statement is another 
numberfield, defined as the closure of $\Q$ in the function field of $\cD_A$. In `practice' it should 
be a rather `small' and explicit extension of $\Q$.

One can amplify the above by considering commuting inertia elements attached to any stratum $\cD_\Tc$
of the divisor at infinity, where $\Tc$ is a (not necessarily maximal) nested set
(see \S \ref{subsec2dpm} above). We only very briefly sketch such an amplification, using the general
notation of \S \ref{subsectangentialbasepoints} (see Equation \eqref{eq:4}). The point is that if $J\subset J'$, we can
construct attending tangential basepoints $\vec\eta_J$ and $\vec\eta_{J'}$. Then there is an
injection $\cD_{J'}\hookrightarrow \cD_J$ (assume $\cD_{J'}$ nonempty) and a (specialization) 
morphism $\vec\eta_{J}\rightarrow \vec\eta_{J'}$ which induces a Galois equivariant morphism
$\pi_1^{geom}(X, \vec\eta_{J})\rightarrow \pi_1^{geom}(X,\vec\eta_{J'})$. In particular one can take 
$J=\{j\}$, a singleton, and $J'$ maximal. 

Let us translate this construction in our case with the  
statement below; variants are possible and we leave these and details to the curious reader.
Start from a nested set $\Sc$, construct the attending basepoint $\vec\eta_\Sc$.
Note that if $\Sc$ is maximal, $\vec\eta_\Sc$ is attached to a (closed) point of the divisor 
$\cD$ and that, correspondingly, it has dimension $=dim(X)$. Then one has

\begin{cor} 
\label{cor:Scyclotomy}

Given $W$ and $\Sc$ a nested set, let $\cD_\Sc\subset \cD$ be the corresponding stratum of the 
divisor at infinity and $d=codim(\cD_\Sc)$. Let $\vec\eta_\Sc$ be an attached tangential basepoint. 
Then there is an inertia group $I_\Sc\subset  \pi_1^{geom}(X_W, \vec\eta_\Sc)$ which is a quotient
of ${\widehat \Z}^d$ and is acted on cyclotomically. In particular, let $k_\Sc$ be the field of definition 
of $\cD_\Sc$ and let $A,B\in \Sc$; then $z_A,z_B \in I_\Sc$, $z_A$ and $z_B$ commute, and 
if $\sigma\in Gal(k_\Sc)$, $\sigma(z_A)=z_A^{\chi(\sigma)}$ (idem $B$).

\hfill $\square$ 

\end{cor}

Note that here we abused notation because the elements $z_A$ and $z_B$ mentioned in
this statement are obtained from those with the same names in Corollary \ref{cor:Acyclotomy}
by changing basepoints, from $\vec\eta_A$ (resp. B) to $\vec\eta_\Sc$. 
All this can be made more concrete using the open affine ${\mathcal U}_\Sc^b$, with functions 
$(u_A)_{A\in\Sc}$. Note also that the field $k_\Sc$
appearing in the statement is the compositum (inside $\bar \Q\subset \C$) of the $k_A$'s, 
$A\in\Sc$, which appear in Corollary \ref{cor:Acyclotomy}, so that $Gal(k_\Sc)$  is a subgroup 
of $Gal(k_A)$. 

Finally we broach, albeit briefly, the case of the {\it full} braid groups $B(W)$, which amounts to 
passing from the inertia elements $z_A$ to the $\zeta_A$'s (see \S \ref{sec:inertia_divisor} and
\S  \ref{sec:quotient}  resp.). We will content ourselves with stating and sketching the proof 
of an analog of Corollary \ref{cor:Acyclotomy}, leaving it to the interested reader to detail and amplify it.
We write $p: {\overline X}_W\rightarrow {\overline Y}_W={\overline X}_W/W$ for the natural
projection. We pick again an irreducible subspace $A\in \Fc_W$, the attending component $\cD_A$ 
of the divisor at infinity of ${\overline X}_W$ and taking up the notation of \S \ref{sec:quotient}, 
we denote by $\cL_A=p(\cD_A)\subset \cL\subset  {\overline Y}_W$ its projection. 

Now we know that $W$ acts freely on $X_W$ so that $Y_W$ is a scheme, indeed an affine scheme.
The action of $W$ is however ramified at infinity, so technically the completion ${\overline Y}_W$ is a
Deligne-Mumford stack, albeit of a fairly `mild' sort (in particular it is globally uniformizable and
generically schematic). An interesting point is with fields of definitions. Let $k=k_A$ denote as above
the field of definition of $\cD_A$, which is a finite extension of $\Q$, and let $\ell_A$ denote 
the field of definition of $\cL_A$. Then $\ell_A$ is nothing but the field of definition of
the divisor $W(\cD_A)\subset \cD$, namely the orbit of $\cD_A$ under the action of $W$.
Because the action of $W$ is rational, $\ell_A$ can be {\it smaller} than $k_A$ (i.e. a strict
subfield of it). However we will not be able to exploit this fact below and will content 
ourselves with working over $k=k_A$. It would be quite interesting, theoretically at least, 
to descend from $k_A$ to $\ell_A$.

Here we simply concoct a tangential basepoint $\vec \xi_A$ based at a geometric generic point 
of $\cL_A$. To this end it is enough to `quotient' the basepoint $\vec\eta_A$ constructed above
by the action of $W_A\subset W$. One could proceed more abstractly, but here the situation is 
particularly concrete; we refer in particular to the proof of Theorem  \ref{thm:center_bw}  above. 
So one way to achieve this is to use an affine patch ${\mathcal U}_{\Sc}^b$ 
(see above Corollary \ref{cor:Acyclotomy}). The quotient algebra ${\mathcal A}_\Sc^b/W_A$ 
is again a polynomial algebra, say ${\mathcal B}_\Sc^b$, and it defines a local affine chart 
${\mathcal V}_{\Sc}^b$ near the generic point of $\cL_A$ (or rather, to be precise, near the generic
point of $\cL_A$ when `stripped' of its possibly nontrivial automorphism group; see \cite[\S 2.2]{lv}). In particular the function 
$v_A=v_A^{1/m}$ belongs to ${\mathcal B}_\Sc^b$. Here $m=m_A=\vert Z(W_A)\vert$
is the ramification index of the projection $p$ along $\cL_A$ (compare the 
proof of Theorem  \ref{thm:center_bw}). So in order to construct the geometric tangential
basepoint $\vec \xi_A$, one considers the partially completed algebra
${\mathcal B}_\Sc^b [ \![ v_A] \! ]$, 
then tensor with $\bar k$ ($\simeq \bar \Q$) in order 
to make it geometric', and proceed as above (see \S \ref{subsectangentialbasepoints}).
The result reads

\begin{cor} 
\label{cor:quotientyclotomy}

Given $W$, $A$, $\cL_A$ as above, together with a tangential basepoint $\vec\xi_A$ based at
a geometric generic point of $\cL_A$, the inertia group 
$${\widetilde I}_A \subset \pi_1^{geom}(Y_W, \vec\xi_A)$$
is procyclic: ${\widetilde I}_A =\langle \zeta_A \rangle$. If $k=k_A$ denotes the field
of definition of an irreducible divisor $\cD_A\subset {\overline X}_W$ sitting above $\cL_A$
($p(W(\cD_A))=\cL_A)$, the action 
of the Galois group $Gal(k)$ on ${\widetilde I}_A $ is cyclotomic
($\sigma(\zeta_A)=\zeta_A^{\chi(\sigma)}$ for $\sigma\in Gal(k))$.

\end{cor}

Note that part of the situation could be unraveled {\it a priori}: since $W$ acts freely
on $X_W$, the action on a formal neighborhood of $\cL_A$ has to be as described above, i.e.
at the generic point of $\cL_A$ one gets a cyclic cover of degree equal to the cardinal of the 
stabilizer (and no less). This simple reasoning could be expanded in order to prove 
the above Corollary, but here we actually have a much more explicit view of the situation.

In conclusion one can ask the following rather natural question, which has already been raised
(see \cite{lo,lv}) \`a propos the moduli stacks or curves and their fundamental groups
i.e. the Teichm\"uller groups (a.k.a. mapping class groups): Is it true that any element 
of a profinite complex braid group (${\widehat P}_W$ or ${\widehat B}_W$) which is acted on 
cyclotomically by the Galois group, is either a torsion element or has a power which belongs to an 
inertia group attached to a stratum of the divisor at infinity of the relevant wonderful model. 
In a few words: Is it true that all the cyclotomic elements are virtually divisorial? 
This would {\it characterize} divisorial inertia in terms of the arithmetic action only.
Note that in the case of the moduli stacks of curves, at least most torsion elements of 
the fundamental groups are indeed acted on cyclotomically (see \cite{lv} as well as \cite{cm}).

\end{document}